\documentclass[11pt]{article}
\usepackage[a4paper,margin=1in]{geometry}
\usepackage{amsmath,amssymb,amsthm,mathtools}
\usepackage{enumitem}

\newtheorem{theorem}{Theorem}[section]
\newtheorem{lemma}[theorem]{Lemma}
\newtheorem{proposition}[theorem]{Proposition}
\newtheorem{corollary}[theorem]{Corollary}
\theoremstyle{definition}
\newtheorem{definition}[theorem]{Definition}
\theoremstyle{remark}
\newtheorem{remark}[theorem]{Remark}

\usepackage[
    final,
    backref=page,
    hyperindex
]{hyperref}
\usepackage[nameinlink,capitalise]{cleveref}

\newcommand{\C}{\mathbb C}
\newcommand{\Z}{\mathbb Z}
\newcommand{\h}{\mathfrak h}

\title{Classification of   \(U(\mathfrak h)\)-Free Modules over the Heisenberg–Virasoro type Lie superalgebra via Polynomial Shift Operators}
\author{Yan Kong, Haibo Chen\textsuperscript{1}, Yucai Su}
\date{}

\begin{document}
\maketitle

\footnotetext[1]{
Corresponding author: Haibo Chen (hypo1025@jmu.edu.cn).
}

\noindent\textbf{Abstract.}
Let \(\mathcal L\) be the Heisenberg–Virasoro type Lie superalgebra whose even part is the twisted Heisenberg--Virasoro algebra and whose odd part is an abelian ideal. We classify \(\mathbb Z_2\)-graded \(\mathcal L\)-modules that are free of super-rank \(1|1\) over \(U(\mathfrak h)=\mathbb C[L_0,F_0]\). Using polynomial shift operators, we express the defining relations as functional equations for the matrix entries of the generators. Assuming the rank one classification for the even subalgebra, we show that, up to parity reversal, every module with nontrivial odd action belongs to a triangular family \(\Omega^+(\lambda,\beta,d)\) or \(\Omega^-(\lambda,a)\), whereas every module with trivial odd action is a direct sum of two rank-one modules for the even subalgebra. We determine the graded isomorphism classes: within the triangular family, the coupling polynomial  \(d\) or \(a\) is invariant up to multiplication by a non-zero scalar. We also describe canonical proper submodules and show that every module in the classification is reducible. In particular, we provide a new realization of the Heisenberg–Virasoro type Lie superalgebra using $2\times2$ matrices with entries associated with  polynomial shift-differential operators.
\vspace{3mm}

\noindent\textbf{Keywords.} Heisenberg–Virasoro type Lie superalgebra;
\(U(\mathfrak h)\)-free module; polynomial shift operator.

\vspace{1mm}
\noindent\textbf{Mathematics Subject Classification (2020).} 17B10, 17B65, 17B68.

\tableofcontents

\section{Introduction}
Throughout this paper, we work over the field $\mathbb{C}$ of complex numbers. Unless otherwise stated, all modules are assumed to be nontrivial. For a Lie algebra $\mathfrak{g}$, its universal enveloping algebra is denoted by $U(\mathfrak{g})$. We use $\mathbb{C}^*$, $\mathbb{Z}$, and $\mathbb{N}$ to denote the sets of nonzero complex numbers, integers, and nonnegative integers, respectively. All vector (super)spaces, (super)modules, and Lie (super)algebras are defined over $\mathbb{C}$.

Modules that are free over the universal enveloping algebra of a
degree-zero subalgebra provide  a natural source of non-weight
representations. Unlike weight modules, the degree-zero generators
act by multiplication on a polynomial ring, while the remaining
generators act through shifts of the polynomial variables. This
description turns representation-theoretic questions into
polynomial identities. Such modules were studied for
\(\mathfrak{sl}_{n+1}\) by Nilsson \cite{8} and, independently, by
Tan and Zhao \cite{9}. Further developments for Lie algebras and
Lie superalgebras can be found in
\cite{CDL,CDL1,15,CG,DL,HCS,LZ,N2,TZ,XZ,YYX1,YYX3}.

Let \(\mathcal L=\mathcal L_{\bar0}\oplus\mathcal L_{\bar1}\)
be the Lie superalgebra considered in this paper. Its even part
\(\mathcal L_{\bar0}\) is the twisted Heisenberg--Virasoro algebra,
and its odd part is an abelian ideal spanned by the elements \(G_m\),
\(m\in\Z\). We study \(\Z_2\)-graded \(\mathcal L\)-modules that
are free of super-rank \(1|1\) over
\[
U(\h)=\C[L_0,F_0],
\qquad
\h=\C L_0\oplus\C F_0.
\]
Thus each homogeneous component is a free rank-one
\(\C[L_0,F_0]\)-module. Our aim is to determine all possible
actions of \(\mathcal L\) on these modules, together with their
graded isomorphism classes and reducibility properties.

Identifying \(L_0\) and \(F_0\) with multiplication by \(x\)
and \(y\), respectively, gives
\[
M\cong\C[x,y]1_{\bar0}\oplus\C[x,y]1_{\bar1}.
\]
The defining brackets force \(L_m\) and \(F_m\) to shift \(x\)
by \(m\), whereas \(G_m\) shifts \(x\) by \(m\) and \(y\) by
\(-1\). The even generators preserve the two homogeneous
components, while the odd generators exchange them. Consequently,
their actions can be expressed by \(2\times2\) matrices of
polynomial shift operators. The remaining defining relations then
become functional equations for the polynomial coefficients.

Using the known classification of rank-one modules over
\(\mathcal L_{\bar0}\), we solve these equations and obtain the
classification up to parity reversal. If the odd ideal acts
nontrivially, the module belongs to one of the triangular families
\[
\Omega^+(\lambda,\beta,d),\quad {\rm or}\quad \Omega^-(\lambda,a),
\qquad
\lambda\in\C^*,\quad
\beta\in\C[y],\quad
0\ne d,a\in\C[y].
\]
If the odd ideal acts trivially, the module is a split module
\[
\Omega^0(\lambda,\beta;
          \widetilde\lambda,\widetilde\beta),
\]
whose two homogeneous components are independent rank-one
\(\mathcal L_{\bar0}\)-modules. We also determine the conditions
under which two modules in these families are graded-isomorphic. In
particular, for fixed \(\lambda\) and \(\beta\), the coupling polynomial
\(d\) is determined up to multiplication by a nonzero scalar. Finally,
we identify canonical proper submodules and show that every
module in the classification is reducible.

The paper is organized as follows: Section~2 introduces
\(\mathcal L\), the \(U(\h)\)-free condition, and the shift
identities. Section~3 develops the polynomial shift operator
description and recalls the rank-one classification for
\(\mathcal L_{\bar0}\). Section~4 constructs the module families and
studies their submodules and graded isomorphisms. Finally, we prove the
classification theorem.

\section{The Heisenberg–Virasoro type Lie superalgebra and \texorpdfstring{$U(\h)$-free modules}{U(h)-free modules}}
In this section, we introduce a class of infinite-dimensional Lie superalgebras known as Heisenberg–Virasoro type Lie superalgebras (see \cite{MM,RS}).

\begin{definition}\label{def:algebra}
Let \(\mathcal{L}=\mathcal{L}_{\bar 0}\oplus \mathcal{L}_{\bar 1}\) be a Heisenberg–Virasoro type Lie superalgebra. The even part \(\mathcal{L}_{\bar 0}\) is spanned by
\[
L_m,\quad F_m,\quad C_i\qquad m\in \Z,\ i=1,2,3,
\]
and the odd part \(\mathcal{L}_{\bar 1}\) is spanned by
\[
G_m,\qquad m\in \Z.
\]
The Lie superalgebra structure is given by the following brackets:
\begin{align}
[L_m,L_n]=&(m-n)L_{m+n}+\delta_{m+n,0}\frac{m^3-m}{12}C_1, \label{eq:def1}\\
[L_m,F_n]=&-nF_{m+n}+\delta_{m+n,0}(m^2+m)C_2,\\
[L_m,G_n]=&(m-n)G_{m+n},\\
[F_m,G_n]=&G_{m+n}, \label{eq:def2}\\
[F_m,F_n]=&m\delta_{m+n,0}C_3,\\
[G_m,G_n]_+=&[\mathcal{L},C_i]=0, \label{eq:def3}
\end{align}
where \(m,n\in\mathbb{Z}\), \(i=1,2,3\).
\end{definition}

The even part is isomorphic to the twisted Heisenberg--Virasoro algebra. The set \(\mathcal{C}={\rm span}\{C_i\mid i=1,2,3\}\) constitutes the center of \(\mathcal{L}\). The centerless twisted Heisenberg--Virasoro algebra and the centerless Heisenberg–Virasoro type Lie superalgebra are defined, respectively, by \(\bar{\mathcal{L}}_{\bar0}=\mathcal{L}_{\bar0}/ \mathcal{C}\) and \(\bar{\mathcal{L}}=\mathcal{L}/ \mathcal{C}\). The displayed relations satisfy the super Jacobi identity, and this convention is adopted throughout. In particular, \([G_m,G_m]_+=2G_m^2=0\).

Assume that \(\mathfrak{G}\) is a Lie superalgebra, and let \(P=P_{\bar0}\oplus P_{\bar1}\) be a \(\Z_2\)-graded vector space. An element \(v\in P_{\bar0}\) is called even, whereas an element \(v\in P_{\bar1}\) is called odd. If \(v\) is even, we set \(|v|=\bar 0\); if \(v\) is odd, we set \(|v|=\bar 1\). Elements of \(P_{\bar0}\) and \(P_{\bar1}\) are called homogeneous. Unless otherwise specified, all elements of superalgebras and modules are assumed to be homogeneous.

An \(\mathfrak{G}\)-module is a \(\Z_2\)-graded vector space \(P\) together with a bilinear map \(\mathfrak{G}\times P\rightarrow P\), denoted by \((a,v)\mapsto av\), such that
\[
a(bv)-(-1)^{|a||b|}b(av)=[a,b]v
\quad \mathrm{and}\quad
\mathfrak{G}_{\bar i} P_{\bar j}\subseteq P_{\bar i+\bar j},
\]
where \(\bar i,\bar j\in\Z_2\), \(a,b\in\mathfrak{G}\), \(v\in P\). Consequently, there is a parity-change functor \(\Pi\) from the category of \(\mathfrak{G}\)-modules to itself. That is, for any module \(P=P_{\bar0}\oplus P_{\bar1}\), we have \(\Pi(P_{\bar 0})=P_{\bar 1}\) and \(\Pi(P_{\bar 1})=P_{\bar0}\).

\subsection{\texorpdfstring{\(U(\h)\)-free modules}{U(h)-free modules}}

Let
$
\h=\C L_0\oplus \C F_0.
$
Since \([L_0,F_0]=0\), the universal enveloping algebra is
$
U(\h)\simeq \C[L_0,F_0].
$
Set
$
x=L_0,y=F_0.
$
Then \(U(\h)\simeq \C[x,y]\).

\begin{definition}\label{def2.1}
An \(\mathcal{L}\)-module \(M\) is called \(U(\h)\)-free of super-rank \(1|1\) if there exist  homogeneous elements
\[
1_{\bar 0}\in M_{\bar 0},\qquad 1_{\bar 1}\in M_{\bar 1}
\]
such that
\[
M=U(\h)1_{\bar 0}\oplus U(\h)1_{\bar 1}.
\]
\end{definition}
Thus, as a vector space,
\[
M\simeq \C[x,y]1_{\bar 0}\oplus \C[x,y]1_{\bar 1}
\simeq \C[x,y]\otimes \C^{1|1}.
\]

We obtain the following shift lemma.

\begin{lemma}\label{lem:shift}
For \(f\in\C[x,y]\), and \(m\in\Z\), we have
\begin{align*}
L_mf(x,y)=&f(x+m,y)L_m,\\
F_mf(x,y)=&f(x+m,y)F_m,\\
G_mf(x,y)=&f(x+m,y-1)G_m.
\end{align*}
\end{lemma}

\begin{proof}
By \eqref{eq:def1}--\eqref{eq:def3}, we obtain
\(L_mL_0=(L_0+m)L_m\), \(F_mL_0=(L_0+m)F_m\),
\(G_mL_0=(L_0+m)G_m\), and \(G_mF_0=(F_0-1)G_m\).
Moreover, \(L_mF_0=F_0L_m\) and \(F_mF_0=F_0F_m\). Induction on the degrees of 
monomial establishes the assertion, and linearity yields the result for \(f\).
\end{proof}

\section{Polynomial shift operators}
Let \(\partial_v=\frac{\partial}{\partial v}\), where \(v\in\{x,y\}\).
For \(m\in \mathbb{Z}\), define the shift operators \(S_m\) and \(T\) on the polynomial algebra \(\mathbb{C}[x,y]\) by
\[
S_m f=f(x+m,y),\qquad T f=f(x,y-1).
\]

\begin{proposition}\label{pro:3.1}
For any \(m\in\mathbb{Z}\), we have \(S_m=e^{m\partial_x}\) and \(T=e^{-\partial_y}\) on \(\mathbb{C}[x,y]\).
\end{proposition}

\begin{proof}
Let \(m\in\mathbb{Z}\) and \(f\in\mathbb{C}[x,y]\). By definition,
\[
S_m f= f(x + m, y).
\]
Expanding \(f(x+m, y)\) in a Taylor series about \(x\), with \(y\) treated as a constant, gives
\[
f(x+m,y)
=\sum_{k\ge0}\frac{m^k}{k!}
\partial_x^k f,
\]
where only finitely many terms are nonzero.
Since \(\partial_x\) is locally nilpotent on \(\mathbb{C}[x,y]\), we obtain
\[
S_m f = f(x+m, y) = e^{m\partial_x} f.
\]
As this identity holds for every polynomial \(f\), it follows that
\[
S_m = e^{m\partial_x}.
\]
Similarly, we get \(T=e^{-\partial_y}\).
\end{proof}

Let \(U_v\)   denote the multiplication operators, i.e., \(U_v f = v f\), where $v\in\{x,y\}$. Then the commutation relations can be expressed as
\[
S_m \circ U_x = (U_x+m)\circ S_m,\ T \circ U_y = (U_y-1)\circ T.
\]
For \(m\in\mathbb{Z}\), the algebra of polynomial shift operators is defined by
\[
\mathcal{D}_{\rm shift}
=
\left\{
\sum^r_{k=0} g_k(x,y)S_{m_k}T^{i_k}
\;\middle|\; m_k\in\mathbb{Z},\ i_k\in\mathbb{N},\
g_k\in \C[x,y]
\right\}.
\]
A representation of \(\mathcal{L}\) on \(M\), as in Definition \ref{def2.1}, is a Lie superalgebra homomorphism
\[
\rho:\mathcal{L}\longrightarrow {\rm Mat}_{1|1}(\mathcal{D}_{\rm shift}),
\]
where \({\rm Mat}_{1|1}(\mathcal{D}_{\rm shift})\) denotes the algebra of \(2\times 2\) matrices with entries in \(\mathcal{D}_{\rm shift}\), equipped with the \(\Z_2\)-grading induced by the parities of the basis vectors. In this paper, we determine the general form of the operators representing the generators of \(\mathcal{L}\). Since \(L_m\) and \(F_m\) are even, they preserve the \(\mathbb Z_2\)-grading of \(M\).

Equivalently, \(M\) is a left \(U(\mathcal{L})\)-module, with the action given by
\[
X \cdot f = \rho(X)f, \qquad X \in \mathcal{L}, \; f \in M.
\]
Consequently, the category of representations of \(\mathcal L\) is equivalent to the category of left \(U(\mathcal L)\)-modules.

In this paper, we do not directly state that \(M\) is a \(U(\mathcal{L})\)-free module, but rather that \(M\) is a \(U(\mathfrak{h})\)-free module. This is because \(\mathfrak{h}\) is very small, \(U(\mathfrak{h}) \simeq \mathbb{C}[x,y]\), so the free condition forces
\[
M \simeq \mathbb{C}[x,y] \otimes \mathbb{C}^{1|1}.
\]

\subsection{The known results}
In this paper, we always assume that \(\lambda \in \mathbb{C}^*\) and \(\beta(y)=\sum_{i=0}^k\beta_iy^i \in \mathbb{C}[y]\). We now recall a family of non-weight \(\mathcal{L}_{\bar0}\)-modules constructed in \cite{HCS}, which are closely related to the classification of \(U(\mathfrak{h})\)-free modules of rank one.

As a vector space, \(\Psi_{\mathcal{L}_{\bar0}}
(\lambda, \beta)= \mathbb{C}[x,y]\) is the polynomial algebra, with the \(\mathcal{L}_{\bar0}\)-action defined by
\begin{equation}\label{eq2.22}
L_m f = \lambda^m (x+ m\beta(y)) f(x+m,y), \quad
F_m f = \lambda^m yf(x+m,y),\quad C_if=0,
\end{equation}
where \(f \in \mathbb{C}[x,y]\), \(m \in \mathbb{Z}\), and \(i=1,2,3\).

The following results were established in \cite{HCS}.

\begin{theorem}\label{tho2.1}
Let \(\lambda,\lambda^\prime\in\mathbb{C}^*\) and \(\beta(y),\beta^\prime(y)\in\mathbb{C}[y]\).
\begin{itemize}
\item[\rm(1)] \(\Psi_{\mathcal{L}_{\bar0}}
(\lambda, \beta)\) is an \(\mathcal{L}_{\bar0}\)-module;
\item[\rm(2)]
\(\Psi_{\mathcal{L}_{\bar0}}
(\lambda, \beta)\cong\Psi_{\mathcal{L}_{\bar0}}
(\lambda^\prime, \beta^\prime)\) if and only if \(\lambda=\lambda^\prime,\beta(y)=\beta^\prime(y)\);
\item[\rm(3)]
Assume that there exists an \(\mathcal{L}_{\bar0}\)-module \(M^\prime\) that is a
\(U(\mathfrak{h})\)-free module of rank one. Then \(M^\prime\cong \Psi_{\mathcal{L}_{\bar0}}
(\lambda, \beta)\).
\end{itemize}
\end{theorem}

\begin{corollary}
By Theorem \ref{tho2.1}, a realization of the centerless twisted Heisenberg-Virasoro algebra \(\bar{\mathcal L}_{\bar0}\) on \(\mathbb{C}[x,y]\) is given by
\begin{align*}
L_{m}=&\lambda^m (x+ m\beta(y)) e^{m\partial_x},  \quad
F_{m}=\lambda^m ye^{m\partial_x}.
\end{align*}
\end{corollary}

\section{Non-weight modules}\label{sec:4}
In this section, we introduce a class of non-weight representations of the Heisenberg–Virasoro type Lie superalgebra.

The following result presents the polynomial shift operator formulation.

\begin{lemma}\label{lem:4.2}
For any \(g,h\in\mathbb{C}[x,y]\) and \(m,n\in\mathbb{Z}\), \(i,j\in\mathbb{N}\), the following operator identity holds on \(\mathbb{C}[x,y]\):
\[
(gS_mT^i)(hS_nT^j)=gh(x+m,y-i)S_{m+n}T^{i+j}.
\]
\end{lemma}

\begin{proof}
Recall that \(S_m\) acts on the \(x\)-variable, namely
\[
S_m f=f(x+m,y).
\]
Let \(f\in\mathbb{C}[x,y]\) be arbitrary. Applying the left-hand side to \(f\), we compute
\begin{align*}
\bigl((gS_mT^i)(hS_nT^j)\bigr)f
=&(gS_mT^i)\bigl(hf(x+n,y-j)\bigr)\\
=&g\bigl(h(x+m,y-i)f(x+m+n,y-i-j)\bigr)\\
=&gh(x+m,y-i)f(x+m+n,y-i-j)\\
=&\bigl(gh(x+m,y-i)S_{m+n}T^{i+j}\bigr)f.
\end{align*}
Since \(f\in\mathbb{C}[x,y]\) was arbitrary, the claimed operator identity follows.
\end{proof}

Lemma \ref{lem:4.2} is the basic calculation behind all the functional equations below.

\begin{proposition}\label{prop:construction}
For any \(m,n\in\mathbb{Z}\), the following three maps  \(\rho\) define Lie superalgebra homomorphisms, thereby yielding \(\mathbb{Z}_2\)-graded representations of the Heisenberg–Virasoro type Lie superalgebra.
\begin{itemize}
\item{\rm(1)}
The explicit Lie superalgebra homomorphism is given by:
\begin{align*}
\rho(L_m)
=&\begin{pmatrix}
\lambda^m(x+m\beta(y))S_m&0\\
0&\lambda^m(x+m(\beta(y-1)+1))S_m
\end{pmatrix}
,\\
\rho(F_m)
=&\begin{pmatrix}
\lambda^m yS_m&0\\
0&\lambda^m yS_m
\end{pmatrix},\\
\rho(G_m)
=&\begin{pmatrix}
0&0\\
\lambda^m d(y)S_mT&0
\end{pmatrix},\\
\rho(C_i)=&\begin{pmatrix}
0&0\\0&0
\end{pmatrix},
\end{align*}
where \(i=1,2,3\), \(\lambda\in\mathbb{C}^*\), \(\beta(y)\in\mathbb{C}[y]\), and \(d(y)\in\mathbb{C}[y]\setminus\{0\}\).  

\item{\rm(2)}
The explicit Lie superalgebra homomorphism is given by:
\begin{align*}
\rho(L_m)
=&\begin{pmatrix}
\lambda^m(x+m(y+1))S_m&0\\
0&\lambda^m(x+my)S_m
\end{pmatrix}
,\\
\rho(F_m)
=&\begin{pmatrix}
\lambda^m yS_m&0\\
0&\lambda^m yS_m
\end{pmatrix},\\
\rho(G_m)
=&\begin{pmatrix}
0&0\\
\lambda^m a(y)(x+my)S_mT&0
\end{pmatrix},\\
\rho(C_i)=&\begin{pmatrix}
0&0\\0&0
\end{pmatrix},
\end{align*}
where \(i=1,2,3\), \(\lambda\in\mathbb{C}^*\),   and \(a(y)\in\mathbb{C}[y]\setminus\{0\}\).

\item{\rm(3)}
The explicit Lie superalgebra homomorphism is given by:
\begin{align*}
\rho(L_m)
=&\begin{pmatrix}
\lambda^m\bigl(x+m\beta(y)\bigr)S_m&0\\
0&\widetilde{\lambda}^m\bigl(x+m\widetilde{\beta}(y)\bigr)S_m
\end{pmatrix},\\
\rho(F_m)
=&\begin{pmatrix}
\lambda^m yS_m&0\\
0&\widetilde{\lambda}^m yS_m
\end{pmatrix},\\
\rho(G_m)=&
\rho(C_i)=\begin{pmatrix}
0&0\\0&0
\end{pmatrix},
\end{align*}
where \(i=1,2,3\), \(\lambda,\widetilde{\lambda}\in\mathbb{C}^*\), and \(\beta(y),\widetilde{\beta}(y)\in\mathbb{C}[y]\). 
\end{itemize}
\end{proposition}

\begin{proof}
For simplicity, we denote \(P_m=x+m\beta(y)\) and \(Q_m=x+m(\beta(y-1)+1)\). It is easy to see that
\[
mX_m-nX_n=(m-n)X_{m+n},
\]
where \(X\in\{P,Q\}\).
Then, by Lemma \ref{lem:4.2}, it suffices to verify the defining commutation relations of the Heisenberg–Virasoro type Lie superalgebra. Since the even elements do not change parity, we need only consider the bracket between \(\rho(L_m)\) and \(\rho(L_n)\) in matrix form:
\begin{align*}
[\rho(L_m),\rho(L_n)]
=& \begin{pmatrix}
\lambda^mP_mS_m & 0\\
0 & \lambda^mQ_mS_m
\end{pmatrix}
\begin{pmatrix}
\lambda^nP_nS_n & 0\\
0 & \lambda^nQ_nS_n
\end{pmatrix} \\
& -
\begin{pmatrix}
\lambda^nP_nS_n & 0\\
0 & \lambda^nQ_nS_n
\end{pmatrix}
\begin{pmatrix}
\lambda^mP_mS_m & 0\\
0 & \lambda^mQ_mS_m
\end{pmatrix} \\
=& \lambda^{m+n}\begin{pmatrix}
P_m(P_n+m)S_{m+n} & 0\\
0 &Q_m(Q_n+m)S_{m+n}
\end{pmatrix} \\
& -\lambda^{m+n}\begin{pmatrix}
P_n(P_m+n)S_{m+n} & 0\\
0 &Q_n(Q_m+n)S_{m+n}
\end{pmatrix}  \\
=& (m-n)\lambda^{m+n}\begin{pmatrix}
P_{m+n}S_{m+n} & 0\\
0 &Q_{m+n}S_{m+n}
\end{pmatrix} \\
=& (m-n)\rho(L_{m+n}).
\end{align*}

Next, we compute the bracket between \(\rho(L_m)\) and \(\rho(F_n)\):
\begin{align*}
[\rho(L_m),\rho(F_n)]
=& \begin{pmatrix}
\lambda^{m}P_mS_m & 0\\
0 & \lambda^{m}Q_mS_m
\end{pmatrix}
\begin{pmatrix}
\lambda^{n}yS_n & 0\\
0 & \lambda^{n}yS_n
\end{pmatrix} \\
&  -
\begin{pmatrix}
\lambda^{n}yS_n & 0\\
0 & \lambda^{n}yS_n
\end{pmatrix}
\begin{pmatrix}
\lambda^{m}P_mS_m & 0\\
0 & \lambda^{m}Q_mS_m
\end{pmatrix} \\
=&-n\lambda^{m+n}\begin{pmatrix}
yS_{m+n} & 0\\
0 &yS_{m+n}
\end{pmatrix} \\
=& -n\rho(F_{m+n}).
\end{align*}

We now show that the bracket between \(\rho(L_m)\) and \(\rho(G_n)\) is given by
\begin{align*}
[\rho(L_m),\rho(G_n)] 
=&\rho(L_m)\rho(G_n)
-\rho(G_n)\rho(L_m) \\
=&\rho(L_m)\begin{pmatrix}
0&0\\
\lambda^n d(y)S_nT&0
\end{pmatrix}
-\rho(G_n)\begin{pmatrix}
\lambda^mP_mS_m & 0\\
0 & \lambda^mQ_mS_m
\end{pmatrix}\\
=&\begin{pmatrix}
\lambda^{m}P_mS_m & 0\\
0 & \lambda^{m}Q_mS_m
\end{pmatrix}\begin{pmatrix}
0&0\\
\lambda^n d(y)S_nT&0
\end{pmatrix}
\\&-\begin{pmatrix}
0&0\\
\lambda^n d(y)S_nT&0
\end{pmatrix}\begin{pmatrix}
\lambda^mP_mS_m & 0\\
0 & \lambda^mQ_mS_m
\end{pmatrix}\\
=&(m-n)\begin{pmatrix}
0&0\\
\lambda^{m+n} d(y)S_{m+n}T&0
\end{pmatrix}\\
=& (m-n)\rho(G_{m+n}).
\end{align*}
The relations \([\rho(F_m),\rho(G_n)]=\rho(G_{m+n})\), \([\rho(F_m),\rho(F_n)]=0\), and \([\rho(G_m),\rho(G_n)]=0\) follow immediately. All remaining cases involving \(C_i\) follow trivially, since \(\rho(C_i)\) is the zero matrix. We have (1)

 By a similar computation as in {\rm(1)}, we obtain  (2) and (3). This completes the proof.
\end{proof}
\begin{remark}
Propositions \ref{pro:3.1} and \ref{prop:construction} provide a new realization of the centerless Heisenberg–Virasoro type Lie superalgebra $\bar{\mathcal{L}}$ by $2\times2$ matrices whose entries are polynomial shift-differential operators, as follows:
\begin{align*}
L_m
=&\begin{pmatrix}
\lambda^m(x+m\beta(y))e^{m\partial_x}&0\\
0&\lambda^m(x+m(\beta(y-1)+1))e^{m\partial_x}
\end{pmatrix}
,\\
F_m
=&\begin{pmatrix}
\lambda^m ye^{m\partial_x}&0\\
0&\lambda^m ye^{m\partial_x}
\end{pmatrix},\\
G_m
=&\begin{pmatrix}
0&0\\
\lambda^m d(y)e^{m\partial_x-\partial_y}&0
\end{pmatrix},
\end{align*}
or 
\begin{align*}
L_m
=&\begin{pmatrix}
\lambda^m(x+m(y+1))e^{m\partial_x}&0\\
0&\lambda^m(x+my)e^{m\partial_x}
\end{pmatrix}
,\\
F_m
=&\begin{pmatrix}
\lambda^m ye^{m\partial_x}&0\\
0&\lambda^m ye^{m\partial_x}
\end{pmatrix},\\
G_m
=&\begin{pmatrix}
0&0\\
\lambda^m  a(y)(x+my)e^{m\partial_x-\partial_y}&0
\end{pmatrix},
\end{align*}
where  \(\lambda\in\mathbb{C}^*\), \(\beta(y)\in\mathbb{C}[y]\),    \(d(y),a(y)\in\mathbb{C}[y]\setminus\{0\}\).  
\end{remark}
Let \(M=M_{\bar 0}\oplus M_{\bar1}\) be an \(\mathcal{L}\)-module that is free of super-rank \(1|1\) as a \(U(\mathfrak{h})\)-module. Thus, with respect to the decomposition \(M=M_{\bar 0}\oplus M_{\bar 1}\), each element of \(\operatorname{End}_{\mathbb C}(M)\) can, by Proposition \ref{prop:construction}, be represented by a \(2\times2\) matrix of operators on \(\mathbb C[x,y]\). 
With respect to the ordered $U(\mathfrak h)$-basis
$(1_{\bar 0},1_{\bar 1})$, we identify
\[
f1_{\bar 0}+g1_{\bar 1}
\quad\text{with}\quad
\binom{f}{g},
\qquad f,g\in\mathbb C[x,y].
\]
Thus, the entries of the following matrices act on the polynomial coordinates $f$ and $g$, while the matrix positions determine the
corresponding homogeneous components.
For any $f,g\in\mathbb{C}[x,y]$ and $m\in\mathbb{Z}$, Proposition \ref{prop:construction} (1) yields the following nontrivial \(\mathcal{L}\)-module actions:
\[
 L_m\cdot\binom{f}{g}
=\rho(L_m)\binom{f}{g}
=
\binom{
\lambda^m(x+m\beta(y))f(x+m,y)
}{
\lambda^m\bigl(x+m(\beta(y-1)+1)\bigr)g(x+m,y)
},
\]
\[ F_m\cdot\binom{f}{g}
=
\rho(F_m)\binom{f}{g}
=
\binom{
\lambda^m y f(x+m,y)
}{
\lambda^m y g(x+m,y)
},
\]
and
\[ G_m\cdot\binom{f}{g}
=
\rho(G_m)\binom{f}{g}
=
\binom{
0
}{
\lambda^m d(y)f(x+m,y-1)
}.
\]
Equivalently, in terms of the homogeneous generators,
\[
\begin{aligned}
L_m\cdot(f1_{\bar0})
 &=\lambda^m(x+m\beta(y))f(x+m,y)1_{\bar0},\\
L_m\cdot(g1_{\bar1})
 &=\lambda^m\bigl(x+m(\beta(y-1)+1)\bigr)
   g(x+m,y)1_{\bar1},\\
F_m\cdot(f1_{\bar0})
 &=\lambda^m y f(x+m,y)1_{\bar0},\\
F_m\cdot(g1_{\bar1})
 &=\lambda^m y g(x+m,y)1_{\bar1},\\
G_m\cdot(f1_{\bar0})
 &=\lambda^m d(y)f(x+m,y-1)1_{\bar1},\\
G_m\cdot(g1_{\bar1})
 &=C_i\cdot(f1_{\bar0})
  =
C_i\cdot(g1_{\bar1})=0, 
\end{aligned}
\]
where \(i=1,2,3\), \(\lambda\in\mathbb{C}^*\),   \(\beta(y)\in\mathbb{C}[y]\) and \(d(y)\in\mathbb{C}[y]\setminus\{0\}\).
 This $\mathcal{L}$-module is denoted by \(\Omega^+(\lambda,\beta,d)\).

Similarly, for any $f,g\in\mathbb{C}[x,y]$ and $m\in\mathbb{Z}$, Proposition \ref{prop:construction} (2) and (3) yield, respectively, the following $\mathcal{L}$-module actions:
\[
\begin{aligned}
L_m\cdot(f1_{\bar0})
 &=\lambda^m(x+m(y+1))f(x+m,y)1_{\bar0},\\
L_m\cdot(g1_{\bar1})
 &=\lambda^m\bigl(x+my\bigr)
   g(x+m,y)1_{\bar1},\\
F_m\cdot(f1_{\bar0})
 &=\lambda^m y f(x+m,y)1_{\bar0},\\
F_m\cdot(g1_{\bar1})
 &=\lambda^m y g(x+m,y)1_{\bar1},\\
G_m\cdot(f1_{\bar0})
 &=\lambda^m a(y)(x+my)f(x+m,y-1)1_{\bar1},\\
G_m\cdot(g1_{\bar1})
 &=C_i\cdot(f1_{\bar0})
  =
C_i\cdot(g1_{\bar1})=0, 
\end{aligned}
\]
and
\[
\begin{aligned}
L_m\cdot(f1_{\bar0})
 &=\lambda^m(x+m\beta(y))f(x+m,y)1_{\bar0},\\
L_m\cdot(g1_{\bar1})
 &=\widetilde{\lambda}^m\bigl(x+m\widetilde{\beta}(y)\bigr)
   g(x+m,y)1_{\bar1},\\
F_m\cdot(f1_{\bar0})
 &=\lambda^m y f(x+m,y)1_{\bar0},\\
F_m\cdot(g1_{\bar1})
 &=\widetilde{\lambda}^m y g(x+m,y)1_{\bar1},\\
G_m\cdot(f1_{\bar0})
 &=
G_m\cdot(g1_{\bar1})
  =C_i\cdot(f1_{\bar0})
  =
C_i\cdot(g1_{\bar1})=0,
\end{aligned}
\]
where \(i=1,2,3\), \(\lambda,\widetilde{\lambda}\in\mathbb{C}^*\),  \(\beta(y),\widetilde{\beta}(y)\in\mathbb{C}[y]\), and   \(a(y)\in\mathbb{C}[y]\setminus\{0\}\).  These $\mathcal{L}$-modules are denoted by \(\Omega^-(\lambda,a)\) and \(\Omega^0(\lambda,\beta;       \widetilde\lambda,\widetilde\beta)\), respectively.

\subsection{Submodules and Isomorphism Classes}

In this subsection, we study the submodule structure and the graded isomorphism
classes of the \(U(\mathfrak h)\)-free modules constructed in Section~4.  

\begin{proposition}\label{prop:4.4}
In \(\Omega^{+}(\lambda,\beta,d)\), the subspaces
$
N=\mathbb C[x,y]1_{\bar 1}, 
N_d=d(y)\mathbb C[x,y]1_{\bar 1}
$
are nonzero proper graded \(\mathcal L\)-submodules, and
$
N_d\subseteq N.
$
Moreover,
$
\Omega^{+}(\lambda,\beta,d)/N\cong \Psi_{\mathcal L_{\bar0}}(\lambda,\beta),
$
and
$
N\cong \Psi_{\mathcal L_{\bar0}}(\lambda,\beta^{\sharp})$, where 
$\beta^{\sharp}(y)=\beta(y-1)+1.
$
Similarly, in \(\Omega^{-}(\lambda,a)\), the subspaces
$
P=\mathbb C[x,y]1_{\bar 1}, 
P_a=a(y)\mathbb C[x,y]1_{\bar 1}
$
are nonzero proper graded \(\mathcal L\)-submodules, and
$
P_a\subseteq P.
$
Moreover,
$
\Omega^{-}(\lambda,a)/P\cong \Psi_{\mathcal L_{\bar0}}(\lambda,y+1),
$
and
$
P\cong \Psi_{\mathcal L_{\bar0}}(\lambda,y).
$
For \(\Omega^{0}(\lambda,\beta;\widetilde\lambda,\widetilde\beta)\), analogous
results hold with the parities interchanged.
\end{proposition}

\begin{proof}
The even generators preserve each homogeneous summand and also preserve \(N_d\)
and \(P_a\), since they do not shift \(y\). Each \(G_m\) annihilates \(N\) and
maps the even summand into \(N_d\) (respectively \(P_a\)). The stated
identifications follow from the explicit action. Reversing parity gives the
assertion for \(\Omega^{0}(\lambda,\beta;\widetilde\lambda,\widetilde\beta)\).
\end{proof}

\begin{theorem}\label{thm:4.5}
No super-rank \(1|1\) \(U(\mathfrak h)\)-free module in Proposition~4.2 is
irreducible. Specifically, every \(\Omega^{+}(\lambda,\beta,d)\) and every
\(\Omega^{-}(\lambda,a)\) is reducible for all admissible parameters, and
every split module \(\Omega^{0}(\lambda,\beta;\widetilde\lambda,\widetilde\beta)\)
is reducible. Consequently, irreducibility cannot be obtained by imposing
additional conditions on \(\lambda\), \(\beta\), \(d\), or \(a\).
\end{theorem}

\begin{proof}
For \(\Omega^{+}(\lambda,\beta,d)\) and \(\Omega^{-}(\lambda,a)\), the result
follows from Proposition~4.4. In a split module, each homogeneous summand is a
nonzero proper submodule. These cases exhaust the classification.
\end{proof}

\begin{theorem}\label{thm:4.6}
Assume that all coupling polynomials $d(y),a(y)$ are nonzero.

\begin{itemize}
\item{\rm (i)}
\(\Omega^{+}(\lambda,\beta,d)\cong \Omega^{+}(\mu,\gamma,c)\) if and only if
$
\lambda=\mu,  \beta=\gamma, d=\alpha c
$
for some \(\alpha\in\mathbb C^{*}\).

\item{\rm (ii)}
\(\Omega^{-}(\lambda,a)\cong \Omega^{-}(\mu,b)\) if and only if
$
\lambda=\mu,  a=\alpha b
$
for some \(\alpha\in\mathbb C^{*}\).

\item{\rm (iii)} The modules \(\Omega^{+}(\lambda,\beta,d)\), \(\Omega^{-}(\mu,a)\), and \(\Omega^{0}(\lambda,\beta;\widetilde\lambda,\widetilde\beta)\) are pairwise non-isomorphic.
\item{\rm (iv)}
Two split modules are isomorphic if and only if their ordered parameters agree:
$$
\Omega^{0}(\lambda,\beta;\widetilde\lambda,\widetilde\beta)
\cong
\Omega^{0}(\mu,\gamma;\widetilde\mu,\widetilde\gamma)
$$
if and only if
$
\lambda=\mu, 
\widetilde\lambda=\widetilde\mu, 
\beta=\gamma, 
\widetilde\beta=\widetilde\gamma.
$
\end{itemize}

Consequently, for fixed \(\lambda\) and \(\beta\), the coupling invariant for
\(\Omega^{+}(\lambda,\beta,d)\) is the projective class
$
[d]\in \bigl(\mathbb C[y]\setminus\{0\}\bigr)/\mathbb C^{*},
$
while for fixed \(\lambda\), the coupling invariant for \(\Omega^{-}(\lambda,a)\) is the
projective class
$
[a]\in \bigl(\mathbb C[y]\setminus\{0\}\bigr)/\mathbb C^{*}.
$
\end{theorem}

\begin{proof}
Let \(\phi\) be an even isomorphism between two modules of type
\(\Omega^{+}(\lambda,\beta,d)\). Since \(\phi\) commutes with \(L_0\) and \(F_0\), it is
\(\mathbb C[x,y]\)-linear. Hence
\[
\phi(1_{\bar 0})=u1'_{\bar 0},\qquad
\phi(1_{\bar 1})=v1'_{\bar 1}
\]
for some \(u,v\in\mathbb C^{*}\). Intertwining \(F_1\) yields
\(\lambda=\mu\), while intertwining \(L_1\) on \(1_{\bar 0}\) yields
\(\beta=\gamma\). Finally,
\[
\phi(G_01_{\bar 0})=G_0\phi(1_{\bar 0})
\]
implies \(u c=v d\), and hence \(d=\alpha c\) with
\(\alpha=u/v\). Conversely, these conditions define an isomorphism. This
proves (i).

The proof of (ii) is analogous: intertwining \(L_1\) and \(F_1\) forces
\(\lambda=\mu\), and intertwining \(G_0\) yields \(a=\alpha b\) with
\(\alpha=u/v\).

For (iii), observe that in \(\Omega^{+}(\lambda,\beta,d)\) and \(\Omega^{-}(\lambda,a)\), the
common kernel \(\bigcap_m\ker G_m\) is the odd summand \(N\), while in a split
module every \(G_m\) vanishes. An isomorphism must preserve this kernel and the
induced even action on \(M/N\). In \(\Omega^{+}(\lambda,\beta,d)\) the even action on \(M/N\) is
\(\Psi_{\mathcal L_{\bar0}}(\lambda,\beta)\), whereas in \(\Omega^{-}(\lambda,a)\) it is
\(\Psi_{\mathcal L_{\bar0}}(\lambda,y+1)\). If these were isomorphic, then by
Theorem \ref{tho2.1} we would have \(\beta(y)=y+1\). But then the odd diagonal action in
\(\Omega^{+}(\lambda,\beta,d)\) becomes
\[
\lambda^m(x+m(\beta(y-1)+1))S_m
=
\lambda^m(x+m(y+1))S_m,
\]
which differs from the odd diagonal action
\(\lambda^m(x+my)S_m\) in \(\Omega^{-}(\lambda,a)\). Hence no such isomorphism exists.
The remaining non-isomorphism claims are immediate from the vanishing of the
odd action in split modules.

For (iv), applying the uniqueness of rank-one modules in Theorem   \ref{tho2.1} separately
to the two ordered homogeneous components yields the result. The theorem
follows.
\end{proof}

\begin{remark}\label{rem:4.7}
Although \(d(y)\) and a nonzero scalar multiple of \(d(y)\) give isomorphic modules,
and similarly for \(a(y)\), a nonconstant \(d(y)\) or \(a(y)\) cannot be
normalized to \(1\), since the only units of \(\mathbb C[x,y]\) are nonzero
constants.
\end{remark}

\begin{remark}\label{rem:4.8}
The original two-sided odd ansatz cannot define a module: if \(a_m,b_m\) are
both nonzero, then
\[
G_m^2 1_{\bar 0}
=
a_m(x,y)b_m(x+m,y-1)1_{\bar 0}\neq 0,
\]
contradicting \([G_m,G_m]_+=2G_m^2=0\). This is the structural reason
triangularity is unavoidable.
\end{remark}

\section[
  Classification of U(h)-free modules of rank $1|1$
]{Classification of \(U(\mathfrak{h})\)-free modules of super-rank \(1|1\)}
In this section, we determine a complete classification of \(U(\mathfrak{h})\)-free modules of super-rank \(1|1\) over the Heisenberg–Virasoro
type Lie superalgebra.

Let $M=M_{\bar 0}\oplus M_{\bar1}$ be an $\mathcal{L}$-module that is free of super-rank \(1|1\) as a $U(\mathfrak{h})$-module, with homogeneous basis elements $u$ and $v$. Clearly, $u$ and $v$ have distinct parities. Let $u=1_{\bar0}\in M_{\bar0}$ and $v=1_{\bar1}\in M_{\bar1}$. By the definition of $\mathcal{L}$, we have $L_0F_0=F_0L_0$. Hence,
$$M=U(\mathfrak{h}) 1_{\bar0}\oplus U(\mathfrak{h})1_{\bar1}=\C[L_0,F_0] 1_{\bar0}\oplus \C[L_0,F_0]1_{\bar1},$$
where $M_{\bar 0}=\C[L_0,F_0] 1_{\bar0}$ and $M_{\bar1}=\C[L_0,F_0]1_{\bar1}$.

We recall the following result from \cite{YYX3}.

\begin{lemma}\label{le311}
Let \(\mathcal{G}=\mathcal{G}_{\bar0}\oplus\mathcal{G}_{\bar1}\) be a Lie superalgebra. Let \(\mathfrak{h}\) be a degree-0 subalgebra of \(\mathcal{G}\) with
\(\mathfrak{h}\subseteq \mathcal{G}_{\bar0}\), and \([\mathcal{G}_{\bar1},\mathcal{G}_{\bar1}]=\mathcal{G}_{\bar0}\). Then there do not exist \(\mathcal{G}\)-modules which are free of rank one as
\(U(\mathfrak{h})\)-modules.
\end{lemma}

In our situation, \(\mathcal L\) has a degree-0 subalgebra
\[
\mathfrak{h}=\mathbb{C}L_0\oplus\mathbb{C}F_0\subseteq\mathcal{L}_{\bar0}.
\]
However, the odd part of \(\mathcal L\) is abelian, so Lemma \ref{le311} does not apply directly. Nevertheless, since \([G_m,G_n]_+=0\), the odd action is necessarily triangular, and  a rank-one free module would force the odd ideal to act trivially and reduce to the even case. Hence the first nontrivial case is super-rank \(1|1\), which we classify below.
For simplicity, we set
$
x=L_0,  y=F_0.
$

We now determine the general form of the operators representing the generators of \(\mathcal{L}\). Since \(L_m\) and \(F_m\) are even, they preserve the \(\mathbb Z_2\)-grading of \(M\). Hence, their representations are diagonal:
\[
\rho(L_m)=
\begin{pmatrix}
A_m(x,y)S_m & 0\\
0 & \widetilde A_m(x,y)S_m
\end{pmatrix},
\]
and
\[
\rho(F_m)=
\begin{pmatrix}
B_m(x,y)S_m & 0\\
0 & \widetilde B_m(x,y)S_m
\end{pmatrix}.
\]
Since \(G_m\) is odd, it reverses parity; therefore, its representation is off-diagonal. By Lemma \ref{lem:shift}, we set
\[
\rho(G_m)=
\begin{pmatrix}
0 &  C_m(x,y)S_mT\\
\widetilde C_m(x,y)S_mT & 0
\end{pmatrix}.
\]
Here,
\[
A_m,\widetilde A_m,B_m,\widetilde B_m,
C_m,\widetilde{C}_m
\in\mathbb C[x,y].
\]

From Theorem \ref{tho2.1}, for any   \(\lambda,\widetilde{ \lambda}\in\mathbb{C}^{*}\) and \(\beta,\widetilde{\beta}\in\mathbb{C}[y]\), we have
\begin{align*}
A_m(x,y)=&\lambda^m(x+m\beta(y)), \\
\widetilde{A}_m(x,y) =& \widetilde{\lambda}^m(x+m\widetilde{\beta}(y)), \\
B_m(x,y)=&\lambda^{m}y, \
\widetilde{B}_m(x,y)  = \widetilde{ \lambda}^{m}y, 
\end{align*}
where \(m\in\mathbb{Z}\).

 We can write
\begin{equation}\label{eq:ab}G_m1_{\bar0}=\widetilde{C}_m(x,y) S_mT  1_{\bar1},\qquad
G_m 1_{\bar1}=C_m(x,y)S_mT 1_{\bar0}.
\end{equation}
By the shift lemma, \eqref{eq:ab} determines the action of \(G_m\) everywhere.

\begin{lemma}\label{lem:oneindex}
For every \(m\), either \(\widetilde{C}_m(x,y)=0\) or \(C_m(x,y)=0\).
\end{lemma}

\begin{proof}
Fix \(k\in\mathbb{Z}\). Since \(\rho(G_k)^2=\begin{pmatrix}
0 &0\\
0 & 0
\end{pmatrix}\), we obtain
\[
\widetilde{C}_k(x,y)C_k(x+k,y-1)=
C_k(x,y)\widetilde{C}_k(x+k,y-1)=0.
\]
Therefore, either \(\widetilde{C}_k(x,y)=0\) or \(C_k(x+k,y-1)=0\). Since \(C_k(x,y)\) is a polynomial, it follows that either \(\widetilde{C}_k(x,y)=0\) or \(C_k(x,y)=0\). Moreover, since \([\rho(F_m),\rho(G_k)]=\rho(G_{m+k})\), we obtain
 \begin{align*}
&\begin{pmatrix}
0 & B_m(x,y)C_k(x+m,y)S_{m+k}T\\
\widetilde{B}_m(x,y)\widetilde{C}_k(x+m,y)S_{m+k}T & 0
\end{pmatrix}
\\&-\begin{pmatrix}
0 & C_k(x,y)\widetilde{B}_m(x+k,y-1)S_{m+k}T\\
\widetilde{C}_k(x,y)B_m(x+k,y-1)S_{m+k}T & 0
\end{pmatrix}
\\=&\begin{pmatrix}
0 & C_{m+k}(x,y)S_{m+k}T\\
\widetilde{C}_{m+k}(x,y)S_{m+k}T & 0
\end{pmatrix}.
 \end{align*}
Thus,
\[
B_m(x,y)C_k(x+m,y)-C_k(x,y)\widetilde{B}_m(x+k,y-1)=C_{m+k}(x,y)
\]
and
\[
\widetilde{B}_m(x,y)\widetilde{C}_k(x+m,y)-
\widetilde{C}_k(x,y)B_m(x+k,y-1)
=\widetilde{C}_{m+k}(x,y).
\]
Hence, for every \(m\in\mathbb{Z}\), either \(\widetilde{C}_m(x,y)=0\) or \(C_m(x,y)=0\).
\end{proof}

\begin{remark}
\label{lem:global-orientation} We have 
  \(G_m1_{\bar0}=G_m1_{\bar1}=0\) for all \(m\), or
  \(G_m1_{\bar1}=0\) for all \(m\) and \(G_m1_{\bar0}\neq0\) for all \(m\), or \(G_m1_{\bar0}=0\) for all \(m\) and \(G_m1_{\bar1}\neq0\) for all \(m\).
\end{remark}

 The case \(\widetilde{C}_m(x,y)=C_m(x,y)=0\) is immediate. We now consider the case \(\widetilde{C}_m(x,y)\neq0\) and \(C_m(x,y)=0\). The case \(C_m(x,y)\neq0\) and \(\widetilde{C}_m(x,y)=0\) follows by reversing the parity. Hence, for any $m\in\mathbb{Z}$, we get 
\[
\rho(G_m)=
\begin{pmatrix}
0 & 0\\
\widetilde C_m(x,y)S_mT & 0
\end{pmatrix}.
\]

\begin{lemma}  \label{lem:solveF}
Under the above assumptions,    \(\tilde{\lambda} = \lambda\), and exactly one of the following holds for each \(n\in \mathbb{Z}\):

\begin{enumerate}[label=\textup{(\roman*)}]
\item There is a nonzero \(d(y)\in \mathbb{C}[y]\) such that
    \[
    \tilde{C}_{n}(x,y) = \lambda^{n}d(y),\qquad \tilde{\beta} (y) = \beta (y - 1) + 1.
    \]
    \item  There is a nonzero \(a(y)\in \mathbb{C}[y]\) such that
    \[
    \tilde{C}_{n}(x,y) = \lambda^{n}a(y)(x + ny),\qquad \tilde{\beta} (y) = y,\qquad \beta (y) = y + 1.
    \]
\end{enumerate}
\end{lemma}

\begin{proof}
By Lemma \ref{lem:oneindex}, the relations
$ 
[\rho(F_m),\rho(G_n)]=\rho(G_{m+n}),
[\rho(L_m),\rho(G_n)]=(m-n)\rho(G_{m+n})
$ 
yield, respectively,
\begin{align}
\widetilde\lambda^m y\widetilde C_n(x+m,y)
-\lambda^m(y-1)\widetilde C_n(x,y)
&=\widetilde C_{m+n}(x,y),
\label{9}\\
\widetilde\lambda^m
(x+m\widetilde\beta(y))\widetilde C_n(x+m,y)
-\lambda^m
(x+n+m\beta(y-1))\widetilde C_n(x,y)
&=(m-n)\widetilde C_{m+n}(x,y).
\label{10}
\end{align}
Substituting \eqref{9} into \eqref{10}, we obtain
\begin{align}\label{11}
\widetilde\lambda^m
\bigl(x+m\widetilde\beta(y)-(m-n)y\bigr)
\widetilde C_n(x+m,y)
=
\lambda^m
\bigl(x+n+m\beta(y-1)-(m-n)(y-1)\bigr)
\widetilde C_n(x,y).
\end{align}
Taking \(n=0\) in \eqref{11}, we have
\begin{align}\label{12}
\widetilde\lambda^m
(x+m\widetilde\beta(y)-my)
\widetilde C_0(x+m,y)
=
\lambda^m
(x+m\beta(y-1)-m(y-1))
\widetilde C_0(x,y).
\end{align}
 Assume
\[
\widetilde C_0(x,y)=\sum_{i=0}^{k}a_i(y)x^i,
\qquad a_k(y)\neq0 .
\]
Comparing the highest coefficient of \(x^{k+1}\) in \eqref{12}, we obtain
\[
\widetilde\lambda=\lambda.
\]
Now  \eqref{12} becomes
\begin{align}\label{14}
(x+mD(y))\widetilde C_0(x+m,y)
=
(x+mE(y))\widetilde C_0(x,y),
\end{align}
where
$ 
D(y)=\widetilde\beta(y)-y,
E(y)=\beta(y-1)-y+1 .
 $
Comparing the coefficient  of \(x^k\) in \eqref{14}, we get
$
km+D(y)m=E(y)m,
$
hence
\begin{align}\label{15}
D(y)-E(y)=-k .
\end{align}

We now show that \(k\leq1\). Suppose, to the contrary, that \(k\ge2\). Comparing the coefficient  of \(x^{k-1}\) in \eqref{14} and using \eqref{15}, we obtain
$
-a_{k-1}(y)
+
mk\left(\frac{k-1}{2}+D(y)\right)a_k(y)=0.
$
Since this identity holds for all \(m\in\mathbb Z\), it follows that
\[
a_{k-1}(y)=0,\quad
\left(\frac{k-1}{2}+D(y)\right)a_k(y)=0 .
\]
As \(a_k(y)\neq0\), we have
\begin{align}   
\label{16}
D(y)=-\frac{k-1}{2}.
\end{align}
Then, by \eqref{15}-\eqref{16}, \(D(y)\) and \(E(y)\) are constants, with
\[
D=-\frac{k-1}{2},\qquad
E=\frac{k+1}{2}.
\]
Thus, \eqref{14} can be written as
\begin{align}   
\label{17}
\left(x-\frac{m(k-1)}{2}\right)
\widetilde C_0(x+m,y)
=
\left(x+\frac{m(k+1)}{2}\right)
\widetilde C_0(x,y).
\end{align}
Comparing the coefficients of $x^{k-2},x^{k-3},\ldots,x$
successively in the above identity, we obtain
\[
a_{k-2}(y)=a_{k-3}(y)=\cdots=a_1(y)=0.
\]
Hence,
\[
\widetilde C_0(x,y)=a_k(y)x^k+a_0(y).
\]
Substituting this expression into \eqref{17} and comparing the coefficient of $x^{0}$ gives
$
(1-k)m^ka_k(y)=2ka_0(y).
$
Since $m$ is arbitrary and $k\ge2$, it follows that $a_k(y)=a_0(y)=0$, a contradiction. Therefore,
\[
\deg_x\widetilde C_0(x,y)\le1 .
\]

We distinguish two cases.

\textbf{Case 1: \(k=0\).}

Then \(\widetilde C_0(x,y)\) is independent of \(x\). Equation \eqref{15} gives
$
D(y)=E(y),
$
and consequently,
$
\widetilde\beta(y)=\beta(y-1)+1 .
$
Applying \eqref{9} with \(n=0\), we obtain
$
\widetilde C_m(x,y)
=
\lambda^m
\bigl(
y\widetilde C_0(x+m,y)
-(y-1)\widetilde C_0(x,y)
\bigr),
$
and hence
\[
\widetilde C_m(x,y)=\lambda^m d(y),
\]
where
$
d(y)=\widetilde C_0(y)\neq0 .
$
This is precisely (i).

\textbf{Case 2: \(k=1\).}

Write
\[
\widetilde C_0(x,y)=a(y)x+b(y).
\]
By \eqref{16},
$
D(y)=0,
$
and thus
$
\widetilde\beta(y)=y .
$
Using \eqref{15}, we obtain
$
E(y)=1,
$
which implies
$
\beta(y-1)-y+1=1,
$
and therefore
$
\beta(y)=y+1 .
$

Finally, applying \eqref{9} with \(n=0\), we have
\[
\widetilde C_m(x,y)
=
\lambda^m \big(a(y)(x+my)+b(y)\big),
\]
where
$
a(y)\neq0.
$
Then \eqref{10} yields $b(y)=0$, proving (ii):
\[
\widetilde C_m(x,y)=\lambda^m a(y)(x+my),
\qquad
\widetilde\beta(y)=y,
\qquad
\beta(y)=y+1 .
\]

Therefore, exactly one of the two possibilities holds for every
\(m\in\mathbb Z\), and the lemma follows.
\end{proof}

\begin{lemma}\label{lem:solveL}
The remaining relations are satisfied automatically by the formulas in Proposition \ref{prop:construction}.
\end{lemma}

\begin{proof}
This follows by direct verification using Lemma \ref{lem:4.2} and \eqref{eq2.22}.
\end{proof}

 The $\mathcal{L}$-module $M$ may be regarded as an $\mathcal{L}_{\bar0}$-module. By Lemmas \ref{lem:shift}, \ref{lem:oneindex}, \ref{lem:solveF}, and \ref{lem:solveL}, we obtain the main results of this paper as follows.
\begin{theorem}\label{thm:main}
Let \(M\) be an \(\mathcal L\)-module whose restriction to a
\(U(\mathfrak h)\)-module is free of super-rank \(1|1\), where
\(\mathfrak h=\mathbb C L_0\oplus\mathbb C F_0\). Then, up to parity
reversal, the following holds.
\begin{itemize}
\item{\rm(1)} If the odd ideal
\(\mathcal L_{\bar 1}=\bigoplus_{m\in\mathbb Z}\mathbb C G_m\)
acts nontrivially,  there exist \(\lambda\in\mathbb C^*\),
\(\beta\in\mathbb C[y]\), and \(0\neq d,a\in\mathbb C[y]\) such that

  $$M\cong
 \Omega^+(\lambda,\beta,d)
  \quad {\rm or}\quad 
  M\cong
 \Omega^-(\lambda,a)
 $$  as \(\mathbb Z_2\)-graded \(\mathcal L\)-modules.
 
\item{\rm(2)} If the odd ideal acts trivially, i.e.,
\(\mathcal L_{\bar 1}M=0\), then there exist
\(\lambda,\widetilde\lambda\in\mathbb C^*\) and
\(\beta,\widetilde\beta\in\mathbb C[y]\) such that $$M \cong
\Omega^0(\lambda,\beta;\widetilde\lambda,\widetilde\beta).$$
 
In this case, the two homogeneous summands are independent rank-one \(\mathcal L_{\bar 0}\)-modules.
\end{itemize}
\end{theorem}

\begin{remark}\label{rem:matrix}
The classification admits a concise matrix formulation. With respect to the basis
\((1_{\bar{0}},1_{\bar{1}})\), the generators of
\(\Omega^{+}(\lambda,\beta,d)\) and \(\Omega^{-}(\lambda,a)\), act respectively as
\[
\rho(L_m)=
\begin{pmatrix}
\lambda^{m}(x+m\beta(y))S_m & 0\\[2pt]
0 & \lambda^{m}\bigl(x+m(\beta(y-1)+1)\bigr)S_m
\end{pmatrix},
\quad
\rho(F_m)=
\begin{pmatrix}
\lambda^{m}yS_m & 0\\[2pt]
0 & \lambda^{m}yS_m
\end{pmatrix},
\]
\[
\rho(G_m)=
\begin{pmatrix}
0 & 0\\[2pt]
\lambda^{m}d(y)S_mT & 0
\end{pmatrix}, 
\]
and
\[
\rho(L_m)=
\begin{pmatrix}
\lambda^{m}(x+m(y+1))S_m & 0\\[2pt]
0 & \lambda^{m}\bigl(x+my\bigr)S_m
\end{pmatrix},
\quad
\rho(F_m)=
\begin{pmatrix}
\lambda^{m}yS_m & 0\\[2pt]
0 & \lambda^{m}yS_m
\end{pmatrix},
\]
\[
\rho(G_m)=
\begin{pmatrix}
0 & 0\\[2pt]
\lambda^{m}a(y)(x+my)S_mT & 0
\end{pmatrix}. 
\]
The split module
\(\Omega^{0}(\lambda,\beta;\widetilde{\lambda},\widetilde{\beta})\)
is characterized by \(\rho(G_m)=0\), with diagonal blocks determined by
\((\lambda,\beta)\) and \((\widetilde{\lambda},\widetilde{\beta})\).
This matrix formulation extends naturally to higher super-rank: even
generators are block diagonal, odd generators occupy off-diagonal
blocks, and the relation \([G_m,G_m]_+=0\) forces the resulting block
matrix to be triangular. Thus, the classification may be interpreted as
an embedding of \(\mathcal{L}\) into a matrix superalgebra over polynomial
shift operators.
\end{remark}

\section*{Acknowledgements}
This work was supported by the National Natural Science Foundation of China (No. 12671036). We sincerely thank Dr. Xiaoyu Zhu for reviewing the manuscript and providing valuable comments and suggestions, particularly regarding Lemma \ref{lem:solveF}.

\section*{Authors' contributions}
All authors contributed equally to this work.

\section*{Data Availability Statement}
This manuscript has no associated data.

\section*{Conflict of Interest}
The authors declare no conflict of interest related to this article.

\bigskip

Yan Kong

\vspace{2pt}
School of Science, Jimei University, Xiamen, Fujian 361021, China

\vspace{2pt}
kongyan616@163.com

\bigskip

Haibo Chen

\vspace{2pt}
School of Science, Jimei University, Xiamen, Fujian 361021, China

\vspace{2pt}
hypo1025@jmu.edu.cn

\bigskip

Yucai Su

\vspace{2pt}
School of Science, Jimei University, Xiamen, Fujian 361021, China

\vspace{2pt}
yucaisu@jmu.edu.cn
\end{document}